\documentclass[11pt,a4paper]{article}

\pdfoutput=1

\usepackage{amsmath,amsfonts,amssymb,amsthm,cite}

\usepackage{graphicx}
\usepackage{xspace}

\evensidemargin=\oddsidemargin
\advance\topmargin by -\headheight
\advance\topmargin by -\headsep

\newcommand{\fetch}[1]{\includegraphics[height=2.5cm]{#1}}
\newcommand{\LEGO}{\textsf{LEGO}\xspace}

\newtheorem{thm}{Theorem}[section]
\newtheorem{lem}[thm]{Lemma}

\newtheorem{cly}[thm]{Corollary}
\newtheorem{prop}[thm]{Proposition}

\theoremstyle{definition}
\newtheorem{defn}[thm]{Definition}

\theoremstyle{remark}
\newtheorem{rem}[thm]{Remark}

\numberwithin{equation}{section}        

\renewcommand{\a}{\alpha}               

\newcommand{\comment}[1]{\textsf{#1}}   

\begin{document}

\thispagestyle{empty}
\quad

\vspace{2cm}
\begin{center}

\textbf{\huge Enumeration of pyramids of one-dimensional
pieces of arbitrary fixed integer length}  

\vspace{1.5cm}

{\large Bergfinnur Durhuus$^{1}$ and S{\o}ren Eilers$^{2}$} \\

\vspace{.5cm} 

Department of Mathematical Sciences, Copenhagen University\\
Universitetsparken 5\\
DK-2100 Copenhagen {\O}, Denmark

\vspace{1cm}

{\large\textbf{Abstract}}\end{center} 
We consider pyramids made of one-dimensional pieces of
fixed integer length $a$ and which may have pairwise overlaps of integer
length from $1$ to $a$. We prove that the number of pyramids of size
$m$, i.e. consisting of $m$ pieces, equals ${am-1\choose
m-1}$ for each $a\geq 2$. This generalises a well known result for $a=2$.
 A bijective correspondence between so-called right (or left) pyramids
and $a$-ary trees is pointed out, and it is shown that asymptotically
the average width of pyramids is proportional to the square root of the size.
 
\vspace{1.5cm}




\vspace{3cm}

\noindent  
$^1$ durhuus@math.ku.dk\\
$^2$ eilers@math.ku.dk

\newpage

\section{Introduction} 

Solutions of the enumeration problem for a variety of \emph{lattice
animals} have been obtained in recent years, although the
problem concerning general animals on the square or the triangular
lattice both remain unsolved. The problems considered in this area are
interesting in their own right from a combinatorial point of view,
some of them being equivalent to
other well known combinatorial problems, and frequently they are
inspired by concrete problems in other fields. Thus, for the standard model of
site percolation on a lattice \cite{grim}, the connected percolation
clusters are lattice 
animals and their combinatorial properties of imminent importance for
the critical properties of the model \cite{congutt}. Viewing the points in a
lattice as centres of the elementary cells of the dual lattice, an
animal can be identified with an edge-connected set of elementary cells on
the dual lattice, also called a \emph{polyomino}. Thus, enumerating
lattice animals and polyominoes amounts to the same problem, although
the motivation for studying a certain class of animals and its 
polyomino counterpart may be quite different. Early enumeration
results for polyominoes can be found in \cite{klarner,temp}. For more
recent results see \cite{bousrech,zeil} and references given there. 

A particular, much studied, problem is that of directed animals 
on a square lattice, first solved by Dhar in 
\cite{dhar}, and later by a number of authors using different methods
\cite{haknad,gouvien,pen,betpen1,betpen2,bdel,shap,bousrech}. A
directed animal on the square lattice is a set of points on the
lattice such that any point in the 
set is the end point of a lattice path starting at the origin all
of whose steps are directed towards either east or north and all of
whose points are contained in the set. The perhaps most elegant
solution to this problem \cite{bousrech} is obtained by observing
\cite{viennot2} that if the 
lattice is rotated through an angle $\pi/2$ counterclockwise and each
point in a directed animal is replaced by a suitable dimer, one
obtains a pyramid of dimers, the detailed definition of which is given
below, restricted such that no dimer is placed directly on top of another. 
The generating function for such pyramids can then be obtained rather
simply as a solution of an algebraic equation. In turn, the solution to
the directed animal problem on the triangular lattice can be obtained
using that those animals correspond to general pyramids of dimers
\cite{viennot2}.  

In this paper we address the problem of enumerating pyramids whose pieces
are of fixed but arbitrary integer length $a$ instead of dimers
(which correspond to $a=2$). These may, of course, be viewed as a
particular type of polyominoes on the square lattice, or one may think
of them as connected, planar LEGOs made of $1\times a$-pieces and
which are obtained by dropping successively pieces from above so that
the resulting configuration is connected.  

More precisely, we shall consider heaps in the sense of
Viennot \cite{viennot} whose basic pieces are one-dimensional and have
fixed integer length $a$ and whose concurrency relation is defined by
assuming each piece to be an interval $]s,s+a[, s\in\mathbb Z$, and
two intervals $\alpha, \beta$ are concurrent if and only if
$\alpha\cap\beta\neq\emptyset$. Thus a \emph{heap}, in this article,
can be thought of as being obtained by dropping a 
finite number of pieces towards a horizontal axis. Recall that
a heap is a \emph{pyramid} if it has a unique bottom piece. We call a
pyramid $p$ a \emph{right $s$-pyramid}, if the bottom piece covers
the interval $]s,s+a[$ and is a leftmost piece in $p$. Similarly, $p$
is a \emph{left $s$-pyramid} if the bottom piece covers
the interval $]s-a,s[$ and is a rightmost piece in $p$.

When using the term pyramid it will henceforth be
assumed, unless otherwise stated, that its bottom piece covers the
interval $]0,a[$. The number of pieces in a pyramid $p$ will be called its
\emph{size} and is denoted by $|p|$.

The main result of this note is the following.

\begin{thm}
\label{thm}
 Given $a\geq 2$, the number of pyramids of size $m$ equals
$am-1\choose m-1$.
\end{thm}

The result is well known for $a=2$ \cite{bousrech}.
We reconsider this case in Section 2 for the purpose of  illustrating
our method of argument which, in particular, involves establishing a bijective
correspondence between pyramids of size $m$ and closed walks on the integers
of length $m$. In Sections 3 and 4 we
generalise in two steps this correspondence to the case $a\geq 3$ from
which the main result will follow.  It is worth noting that to obtain
this result we do not
rely on generating function techniques whose applicability seems to be
restricted to the dimer case. Those techniques, on the other hand, are
used to determine the asymptotic
behaviour of the average width of pyramids of large size in
Proposition~\ref{averasymp}. We also
point out a bijective correspondence between right pyramids of size $m$ and
$a$-ary trees with $m$ nodes.  In Section 5 we conclude with some
numerical results and comments concerning the
growth rate of the number of general planar LEGOs as a function of 
size and on the dependence of the exponential growth constant on the
size of the building blocks.

\section{The dimer case}
\label{sec2}

 In this section we assume $a=2$. Hence the pieces in this case can be
thought of as dimers. 

We first note the following decomposition property.

\begin{lem}
\label{decomp1}
There is a bijective correspondence between pyramids of size $m\geq 1$ and
sequences $(p_1,p_2,\dots,p_r)$ of pyramids such that $p_i$ is a right
$0$-pyramid if $i$ is odd and a left $1$-pyramid if $i$ is even, and
such that $|p_1|+\dots +|p_r|=m$.
\end{lem}

\begin{proof} If the pyramid $p$ is not a right $0$-pyramid there is a
lowest piece in $p$ above the interval $]-1,1[$ and this piece is
the bottom piece of a unique proper sub-pyramid $p'$ and we can write 
$p=p_1\odot p'$, with notation as in \cite{viennot}, where $p_1$ is a
right $0$-pyramid. If $p'$ is not a left $1$-pyramid it contains a
unique lowest
piece above the interval $]0,2[$ and we have $p'=p_2\odot p''$, where
$p_2$ is a left $1$-pyramid. Repeating the argument the claim follows.
See also \cite{bousrech} for a similar decomposition.
\end{proof}

\begin{defn}
\label{postring}
A finite sequence of $0$'s and $1$'s will be called a \emph{string}
and by a \emph{$(n,m)$-string} we mean a string of length $n$ with $m$
$1$'s. A $(2m,m)$-string $x_1x_2\dots x_{2m}$ is called
\emph{positive} if 
$$
t_s \equiv \sum_{u=1}^s (2x_u -1) 
$$ 
is non-negative for all $s=1,\dots,2m$, i.e. the number of $0$'s in
$x_1\dots x_s$ at most equals the number of $1$'s in $x_1\dots x_s$ for
each $s$.
\end{defn}

Note that a positive $(2m,m)$-string necessarily begins with a $1$ and
ends with a $0$. There is a natural correspondence between strings and
nearest neighbouring walks on the integers starting at $0$ where each
$0$ corresponds to a left-step and each $1$ to a right-step. 
Positive strings then correspond to walks on the non-negative integers starting at $0$.

\begin{lem}
\label{stringpyr1}
There is a bijective correspondence between positive $(2m,m)$-strings
and right $0$-pyramids.
\end{lem}
\begin{proof}
Let $p$ be a right $0$-pyramid of size $m$. We construct inductively the
corresponding positive string $x_1\dots x_{2m}$ together with a sequence
$p^{(1)}\dots p^{(2m)}$ of $0$-pyramids that are sub-pyramids of $p$
such that $p^{(2m)}=p$ as follows. 

Let $x_1=1$ and $p^{(1)}$ be the bottom piece of $p$. Assume $x_1\dots
x_s$ and $p^{(1)}\dots p^{(s)}$ have been constructed. If a piece 
 above the interval $]t_s,t_s+1[$ can be dropped onto $p^{(s)}$ to
obtain a sub-pyramid of $p$ we let $p^{(s+1)}$ be that pyramid and set
$x_{s+1}=1$. Otherwise, set $p^{(s+1)}= p^{(s)}$ and $x_{s+1}=0$. Here
$t_s$ is given as in Definition~\ref{postring} and one readily checks
that at any stage $t_s$ is less than the width of $p^{(s)}$,
i.e. the length of the projection of $p^{(s)}$ onto the horizontal axis, and
that the size of $p^{(s)}$ equals the number of $1$'s in $x_1\dots
x_s$. Indeed, by construction, any piece that can be dropped onto
$p^{(s)}$ to obtain a sub-pyramid of $p$ is above some interval
contained in $]0,t_s+1[$, and at any stage we choose the rightmost of
those pieces to obtain $p^{(s+1)}$. It follows that the so obtained
sequence $x_1\dots x_{2m}$ after $2m$ steps is a positive
$(2m,m)$-string since otherwise the number of $1$'s would be less than
$m$ and $t_{2m}$ would hence be negative, which is not possible.

If $p\neq p'$, the corresponding sequences $p^{(1)}\dots p^{(2m)}$ and  
$p'^{(1)}\dots p'^{(2m)}$ will deviate at some minimal step $s$,
$1<s\leq 2m$, and it follows that the corresponding strings also deviate at
step $s$. 
On the other hand, any positive $(2m,m)$-string $x_1\dots x_{2m}$ can
be obtained by the described procedure from the right $0$-pyramid $p$
obtained by successively dropping a piece above those intervals
$]t_s,t_s+1[$ for which $x_s=1$, with the convention $t_0=0$. This
concludes the proof. 
\end{proof}

\begin{lem}
\label{postringpyr1} Any $(2m,m)$-string $w=x_1\dots x_{2m}$ starting
with $x_1=1$ can be written in a unique way by juxtaposition as
$$
w= w_1\dots w_r
$$
where $w_i$ is a positive $0$-string if $i$ is odd, and $w_i^{-1}$ is
a positive $0$-string if $i$ is even. Here $w_i^{-1}$ denotes the
string obtained from $w_i$ by reversing its order.
\end{lem} 
\begin{proof} Using the correspondence between strings and nearest
neighbouring walks on the integers, we see that the statement amounts  
to asserting the obvious unique decomposition of a walk starting and
ending at $0$ into an alternating sequence of walks on the
non-negative, respectively the non-positive, integers.
\end{proof}

We are now in a position to derive the following result which, in
particular, proves Theorem~\ref{thm} in case $a=2$. 

\begin{prop}
\label{enuma2}
There is a bijective correspondence between pyramids of size $m$ and
$(2m,m)$-strings starting with $1$. In particular, the number of
pyramids of size $m$ equals $2m-1\choose m-1$.
\end{prop}
\begin{proof} Since, obviously, there is a bijective correspondence
between left and right $0$-pyramids of given size and since reversal
of ordering of a string is injective, the claimed correspondence
follows from the preceding three lemmas. The last statement follows
by noting that a $(2m,m)$-string starting with $1$ is uniquely
determined by the position of the remaining $1$'s among the remaining
$2m-1$ entries of the string.
\end{proof}

\begin{figure}\label{sixconfs}
\begin{center}
\begin{tabular}{|c|c|c|c|c|}\hline
\fetch{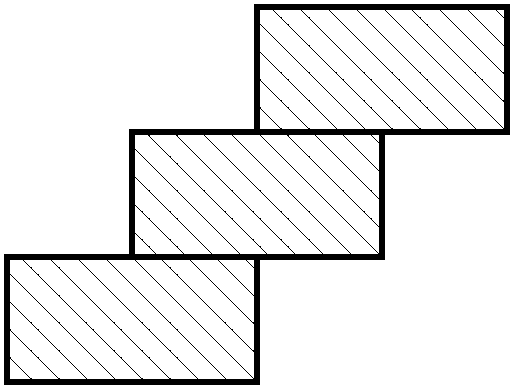}&\fetch{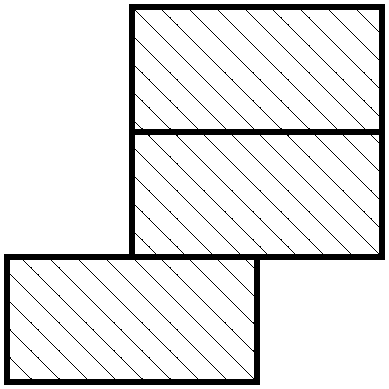}&\fetch{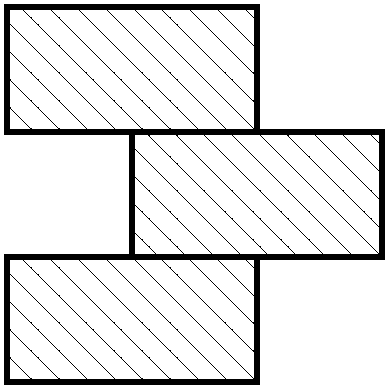}&\fetch{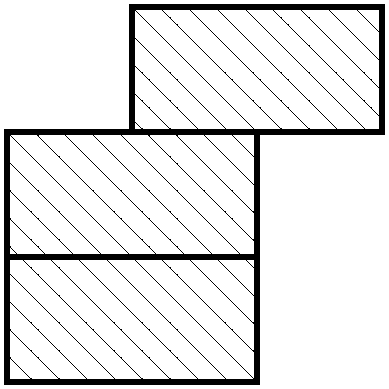}&\fetch{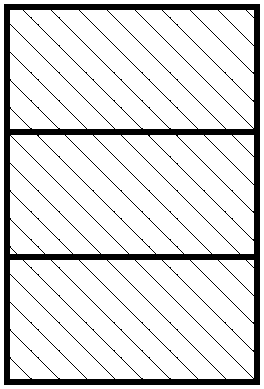}\\\hline
$p_1=111000$&$p_1=110100$&$p_1=110010$&$p_1=101100$&$p_1=101010$\\\hline
\end{tabular}\\
\begin{tabular}{|c|c|c|c|c|}\hline
\includegraphics[height=1.66cm]{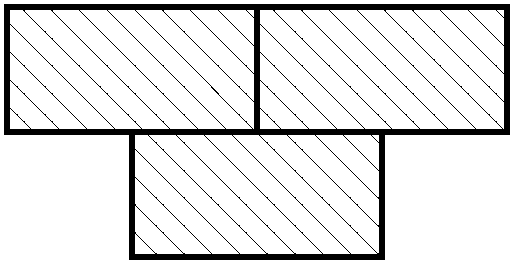}&\fetch{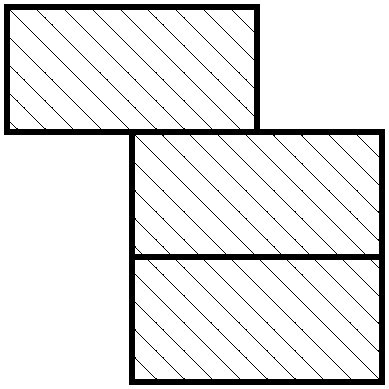}&\fetch{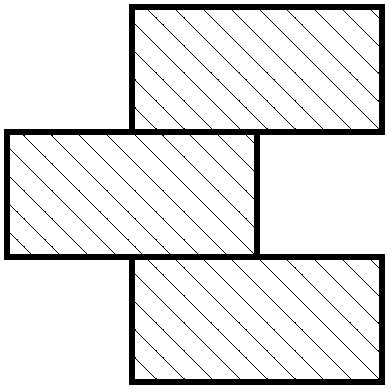}&\fetch{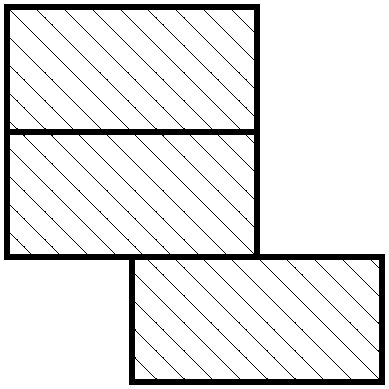}&\fetch{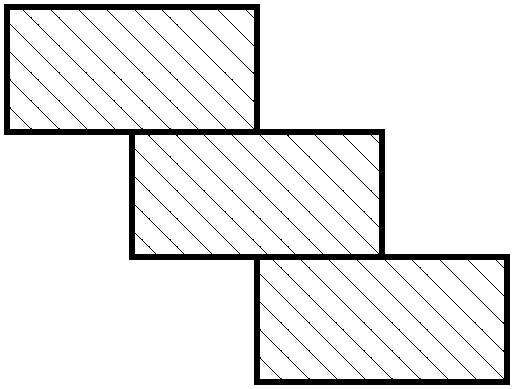}\\\hline
$p_1=1100$&$p_1=1010$&$p_1=10$&$p_1=10$&$p_1=10$\\
$p_2={01}$&$p_2={01}$&$p_2={01}$&$p_2={0101}$&
$p_2={0011}$\\
&&$p_3=10$&&\\\hline
\end{tabular}
\end{center}
\caption{The ten $2$-pyramids of size 3 and their description as $(6,3)$-strings}
\end{figure}

\section{Decomposition of pyramids and strings for $a\geq 3$}
\label{sect3}

In this and the subsequent section we assume $a$ is fixed and larger
than or equal to $3$. Accordingly, we extend the notion of positive
strings as follows. 
\begin{defn}
\label{postringa}
 An $(am,m)$-string $x_1x_2\dots x_{am}$ is called \emph{positive} if 
$$
t_s \equiv \sum_{u=1}^s (a\,x_u -1) 
$$ 
is non-negative for all $s=1,\dots,am$. Moreover, a string is called
\emph{negative} if the reversed string is positive.
\end{defn}

Note that a positive string necessarily begins with $1$ and ends
with at least $a-1$ consecutive $0$'s. 

The following two lemmas are simple generalisations of
Lemmas~\ref{decomp1} and \ref{stringpyr1}.  

\begin{lem}
\label{decomp2}
There is a bijective correspondence between pyramids of size $m\geq 1$ and
sequences $(p_1,p_2,\dots,p_r)$ of pyramids such that $p_i$ is a right
$s_i$-pyramid if $i$ is odd and a left $s_i$-pyramid if $i$ is even,
$|p_1|+\dots +|p_r|=m\,,\; s_1=0,\,$ and 
$$
1\leq s_{i+1}-s_i \leq a-1\quad \mbox{if $i$ is odd}\qquad
\mbox{and}\qquad 1\leq s_i-s_{i+1} \leq a-1\quad \mbox{if $i$ is even.}
$$
\end{lem}
\begin{proof}
 If the pyramid $p$ is not a right $0$-pyramid there is a
lowest piece in $p$ above some interval that overlaps the negative real
axis, that is an interval $]s_2-a,s_2[$, where $0<s_2<a$. This piece is
the bottom piece of a unique proper sub-pyramid $p'$ and we can write 
$p=p_1\odot p'$, where $p_1$ is a
right $0$-pyramid. If $p'$ is not a left $s_2$-pyramid it contains a
unique lowest piece above some interval $]s_3,s_3+a[$, where
$s_2-a<s_3<s_2$. This piece is the bottom piece of a proper sub-pyramid
$p''$ of $p'$, and we have $p'=p_2\odot p''$, where $p_2$ is a left
$s_2$-pyramid. The claim follows by repeating the argument a
sufficient number of times. 
\end{proof}
\begin{lem}
\label{stringpyr2}
There is a bijective correspondence between right $0$-pyramids of
pieces of length $a$ and of size
$m$ and positive $(am,m)$-strings.
\end{lem}
\begin{proof}
The claim follows by a straight-forward generalisation of the proof of
Lemma~\ref{stringpyr1} the details of which are left to the reader.
\end{proof}

The following related correspondence between right pyramids and trees
will not be needed in the proof of Theorem~\ref{thm}, but may be of some
independent interest. Recall that an \emph{n-ary tree}, where $n$ is a
fixed positive integer, is a planar rooted tree all of whose vertices
have order $1$ or $n+1$ and whose root has order $1$. The vertices of
order $n+1$ are called \emph{nodes}.

\begin{prop}
\label{pyratrees}
There is a bijective correspondence between right $0$-pyramids of
pieces of length $a$ and of size
$m$ and $a$-ary trees with $m$ nodes.
\end{prop}
\begin{proof}
   First, note that there is an
obvious bijective correspondence between $(am,m)$-strings and walks on
the integers starting and ending at $0$ and consisting of $m$
right-steps, each of length $a-1$, and $(a-1)m$ left-steps, each of
length $1$. In fact, $t_s$ as given in Definition~\ref{postringa} defines
the $s$'th site visited by the walk corresponding to a given
string. Alternatively, the corresponding walk can
 be viewed as a path on the square lattice $\mathbb Z^2$
from $(0,0)$ to $(am,0)$ with steps $(1,a-1)$ or $(1,-1)$, called
\emph{up-steps} and \emph{down-steps}, respectively. Positive
strings then correspond to paths with vertices on or above the first
axis only, and they are called \emph{generalised Dyck $(a-1)$-paths}
\cite{cameron}, \cite{cameronii}.  

Consider a generalised Dyck $(a-1)$-path $\omega$ and let $\omega'$
be the path obtained by removing the first step, which is necessarily
an up-step. Thus $\omega'$ starts at height $a-1$ and ends at height
$0$. Let now $\omega_1$ be the part of $\omega'$ extending from the
the first vertex in $\omega'$ at height $0$ to the 
final vertex $(0,0)$. Then $\omega_1$ is a generalised $(a-1)$-Dyck
path (possibly trivial), and $\omega'$ equals a path $\omega''$
starting at level $a-1$, ending at level $1$ and nowhere dropping
below level $1$, followed by first a
down-step and then by $\omega_1$. Next, let $\omega_2$ be the part of
$\omega''$ extending from the first
vertex in $\omega''$ at height $1$ to the final vertex $(x_1,1)$. Then
$\omega_2$ is a translated generalised Dyck $(a-1)$-path and the
construction may be repeated $a$ times to yield a decomposition of
$\omega$ into a sequence $\omega_a,\omega_{a-1},\dots,\omega_1$ of
generalised $(a-1)$-Dyck paths (suitably 
translated and possibly trivial) connected by single down-steps and
preceded by an up-step. As a consequence, the number $A_m$ of
generalised $(a-1)$-Dyck paths with $m$ up-steps satisfies the
recursion relation
$$
A_m = \sum_{{m_1+\dots +m_a=m-1}\atop{m_1,\dots,m_a\geq
0}}A_{m_1}\cdot\ldots\cdot A_{m_a}
$$
Rephrased in terms of the generating function 
\begin{equation}
\label{genfct}
A(t) = \sum_{m=1}^\infty A_m\,t^m
\end{equation}
this relation takes the form
\begin{equation}
\label{recursion}
A(t) = t(1+ A(t))^a\,.
\end{equation}
This identity has been noted previously in \cite{cameron} (for $a=3$).
It is well known, and easy to establish, that the generating function
for the number of $a$-ary trees as a function of the number of nodes
likewise satisfies \eqref{recursion}. We conclude that the
coefficients are equal and hence, in view of Lemma~\ref{stringpyr2},
the claimed bijection is established.
\end{proof}
\begin{cly}
\label{cor}
The number $A_m$ of right pyramids of size $m$ is given by
$$
A_m = \frac{1}{(a-1)m+1}{am\choose m} =
\frac{(am)!}{m!((a-1)m+1)!}\,\quad m\geq 1.
$$
\end{cly}
\begin{proof}  It is known that
the stated expression for $A_m$ equals the number of $a$-ary trees
with $m$ nodes, see e.g. \cite{stanley}. Hence the claim follows
from Proposition~\ref{pyratrees}.
\end{proof}
 
\begin{rem}
It is, in fact, quite straight-forward to show directly, by a slight
modification of the argument given in the proof of Proposition 1 in
\cite{bousrech}, that the generating function for the number of right
$0$-pyramids as a function of size satisfies the identity
\eqref{recursion}. The argument involving generalized Dyck paths given
above exhibits at the same time a 
proper generalisation of the well-known, and much exploited, correspondence
between binary trees and standard Dyck paths, see e.g. \cite{LeGall}.  
\end{rem}

\smallskip 

In order to continue our efforts to establish a decomposition result
analogous to Lemma~\ref{postringpyr1} some additional
notation will be needed. We
shall find it convenient to use the language of walks
instead of strings in the following. Hence all walks subsequently will
be assumed to have right-steps of length $a-1$ and left-steps of
length $1$. A generic walk starting at $i\in\mathbb Z$ and ending
at $j\in\mathbb Z$ will be denoted by $S_{ij}(m), S'_{ij}(m')$ etc. For
$i=0$ the walk corresponding to an $(n,m)$-string is obtained by
letting $t_s$ given as in Definition~\ref{postringa} be its $s$'th
site. In Figures \ref{grafN} and \ref{grafT} and walks are illustrated by paths in $\mathbb
Z^2$, replacing each right-step by an up-step $(1,a-1)$ and each
left-step by a down-step $(1,1)$. In the case illustrated, $a=6$.

Given two walks $S'_{ij}(m')$ and $S''_{jk}(m'')$, the walk obtained by
by first traversing $S'_{ij}(m')$ and then $S''_{kj}(m'')$ will be
called the walk obtained by \emph{composing} $S'_{ij}(m')$ and $S''_{jk}(m'')$
and will be denoted by $S'_{ij}(m')S''_{jk}(m'')$. Thus, composition of
walks corresponds to juxtaposition of the corresponding strings. 

Evidently, an $(am,m)$-string is positive if and only if the
corresponding walk takes place on the non-negative
integers. Generally, we shall call a walk $S_{ij}(m)$ positive if
$j\geq i$ and  the walk $S_{ii}(m)$ obtained by adding $j-i$
left-steps at the end is a translate (by $i$) of a walk on the
non-negative integers. Positive walks will be denoted by
$P_{ij}(m), P'_{ij}(m')$ etc. 

Given a walk $S_{ij}(m)$, its \emph{inverse walk} $S^{-1}_{j\,(2j-i)}(m)$ is
defined as the walk obtained by reflecting $S_{ij}(m)$ in the point
$j$ and reversing its direction of traversal. If $i=j=0$ this
corresponds to reversing the order of the corresponding string. A
walk is called \emph{negative} if its inverse is a positive walk. 
Generic negative walks will be denoted by $N_{ij}(m), N'_{ij}(m')$ etc. 
Note that positive walks begin with a right-step whereas negative
walks end with a right-step.

For $0\leq j<i\leq a-2$ we denote by $T_{ij}$ the straight walk from
$i$ to $j$ consisting of $i-j$ left-steps and,  for $0\leq j<k\leq
a-2$, we define
$$
S_{ij}(m) = S'_{ik}(m)T_{kj} \Leftrightarrow S'_{ik}(m) =
S_{ij}(m)U_{jk}\,,
$$
that is the last equation means that $S_{ij}(m)$ contains at least $k-j$
consecutive left-steps at the end, and $S'_{ik}(m)$ is obtained from
$S_{ij}(m)$ by deleting its last $k-j$ steps. In particular, we note
that any walk $P_{ii}(m)$ necessarily ends with at least $a-1$
left-steps such that  $P_{ii}(m)U_{ik}$ is well-defined, and ends with
a left-step, for $0\leq i<k\leq a-2$.

\begin{defn}\label{admis}
Consider a multiple composition of walks of the types $P_{ii}(m),\;
 N_{ii}(m),\; 0\leq i\leq a-2,$ and $T_{ij},\; 0\leq j<i\leq a-2,$ and
with possible insertions of terms $U_{ik},\; 0\leq i<k\leq a-2.$ By
dropping the endpoint indices $i,j,k$ and the step numbers $m$ in the
composition we obtain a word in the alphabet $P,N,T,U$. The
composition is called \emph{admissible} if only neighbouring pairs of
letters of the form 
\begin{equation}\label{pairs}
PN,\; NP,\; PT,\; TP,\; NT,\; TN,\; PU,\; UN
\end{equation}
occur in the corresponding word.
\end{defn}  

We are now in a position to formulate and prove the desired
decomposition result for walks.
\begin{lem}\label{postringpyr2} Any walk $S_{0j}(m)$, where $0\leq j\leq
a-2$, can be written in a unique way as an admissible composition.
\end{lem}

\begin{proof} Let the walk $S_{0j}(m)$ be given and let us consider the
two possible options, depending on the direction of its last step, separately.

\bigskip

i)~~ If $S_{0j}(m)$ ends with a right-step, any composition of the
claimed type must be of the form
$$
S_{0j}(m) = S'_{0j}(m')N_{jj}(m-m')\,,
$$
where $S'_{0j}(m')$ is either empty or is a composition that ends with
a left-step, since 
walks corresponding to pairs in the list \eqref{pairs} not ending with
$N$ must end with a left-step. On the other hand, there evidently exists a
unique $S'_{0j}(m')$ ending with a left step such that the
decomposition above holds.

\bigskip

ii)~~  If $S_{0j}(m)$ ends with a left-step, any composition of the
claimed type must end with a $P_{jj}(m')$, a $T_{ij}$ or a
$U_{ij}$. 

If it ends with a $P_{jj}(m')$ and is not positive, the composition
must have one of the following three forms:   
\begin{eqnarray*}
S_{0j}(m) &=&
S'_{0j_1}(m_0)N_{j_1j_1}(m_1)T_{j_1j_2}P_{j_2j_2}(m_2)T_{j_2j_3}P_{j_3j_3}(m_3)
\dots T_{j_rj}P_{jj}(m')\,,\quad r\geq 1\,,\\
S_{0j}(m) &=&
S'_{0j_1}(m_0)N_{j_1j_1}(m_1)P_{j_1j_1}(m_2)T_{j_1j_2}P_{j_2j_2}(m_2)T_{j_2j_3}P_{j_3j_3}(m_3)
\dots T_{j_rj}P_{jj}(m')\,,\quad r\geq 1\,,\\
S_{0j}(m) &=&
S'_{0j}(m_0)N_{jj}(m_1)P_{jj}(m')\,,
\end{eqnarray*}
where $S'_{0j}(m_0)$ is either empty or a composition of the claimed
form, because $P$ can only be preceded by $T$ or $N$ and $T$ can only be
preceded by $P$ or $N$. Setting $j_1=j$ in the last case it is seen
that in all three cases the last step in $N_{j_1j_1}(m_1)$ is the last (right)
step, call it $\alpha$, in $S_{0j}(m)$ whose initial point has
negative value and whose final point is non-negative, and hence
belongs to $\{0,1,\dots,a-2\}$. In case  $S_{0j}(m)$ is positive we
must have $j=0$ and  $S_{0j}(m)= P_{jj}(m)$.

That the $P$-and $T$-terms occurring in these compositions are
uniquely determined can be seen as follows. Consider the step $\alpha$
defined above with endpoint $j_1$. If the subsequent step is a
left-step there exists a $j_2<j_1$ such that $\alpha$ is followed by
$T_{j_1j_2}$ and then by a right-step. This right-step is the initial
step of a unique positive walk $P_{j_2j_2}(m_2)$ that is followed by a
left-step, unless it equals $P_{jj}(m')$. If not, the argument can 
then be repeated. If the first step after $\alpha$ is a right-step, it
is the initial step of a unique positive walk $P_{j_2j_2}(m_2)$ that is
followed by a left-step, unless it equals $P_{jj}(m')$. Now continue
as previously until all $P$- and $T$-terms have been determined. 

In case the composition ends with a $T_{ij}$ or a $U_{ij}$,
essentially the same argument can be applied to establish
uniqueness of the factors subsequent to the step $\alpha$ defined
above. If $\alpha$ does not exist, i.e. if $S_{0j}(m)$ is positive,
the unique composition of the claimed type must in this case be
$S_{0j}(m) = P_{00}(m)U_{0j}$.

To establish existence of the composition for the part of $S_{0j}(m)$
subsequent to $\alpha$ one can proceed along the same lines just
explained concerning uniqueness. Indeed, if $\alpha$ is followed by a
left-step there must 
exist a non-negative $j_2<j_1$ such that $\alpha$ is followed by
 $T_{j_1j_2}$, which is then followed by a right-step. This right-step is
the first step of a $P_{j_2j_2}(m_2)$. Choosing $m_2$ maximal, it
follows that $P_{j_2j_2}(m_2)$ is either followed by a 
left-step, in which case the construction can be repeated, or the 
end of $P_{j_2j_2}(m_2)$ coincides with that of $S_{0j}(m)$, in which
case the construction is finished, or the end of $P_{j_2j_2}(m_2)$
exceeds that of $S_{0j}(m)$ by $j-j_2>0$ left-steps, in which case these
are annihilated by inserting $U_{j_2j}$ at the end. 

The case where $\alpha$ is followed by a right-step is treated in the
same way.

Together, i) and ii) prove the assertion of the lemma by induction. 
\end{proof}

\begin{figure}
\begin{center}
\includegraphics[width=11cm]{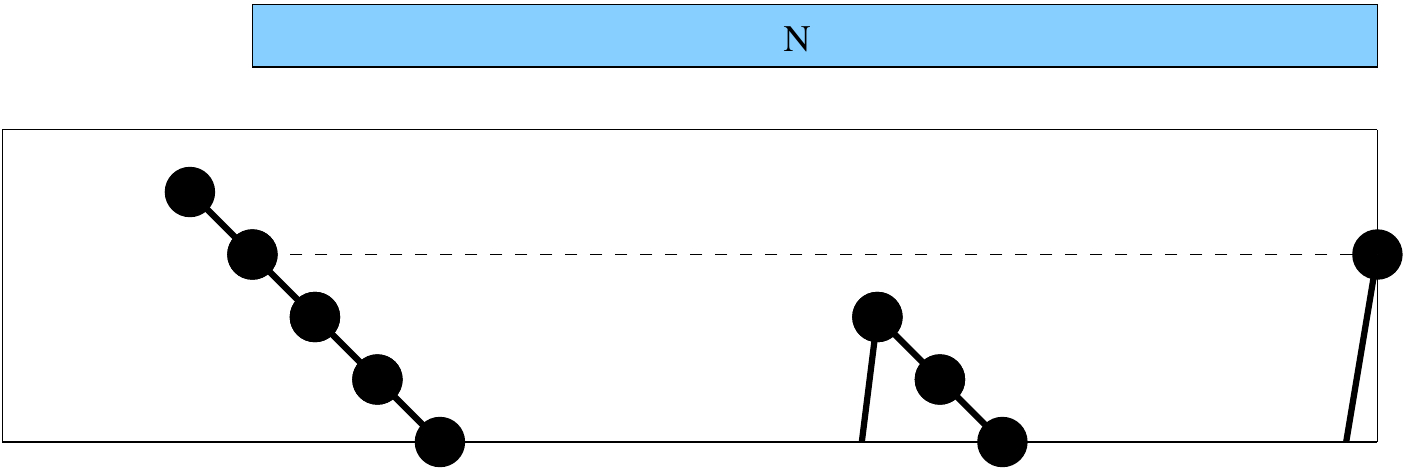}
\end{center}
\caption{Case i) of Lemma \ref{postringpyr2}: String ending in $N$}\label{grafN}
\end{figure}

\begin{figure}
\begin{center}
\includegraphics[width=16cm]{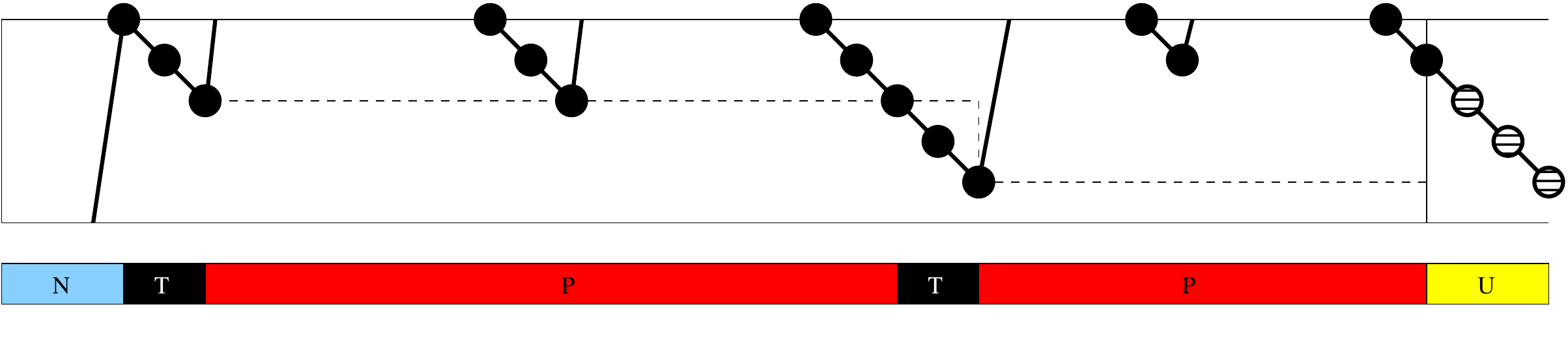}
\end{center}
\caption{Case ii) of Lemma \ref{postringpyr2}: String ending in $PU$}\label{grafT}
\end{figure}

\begin{cly}\label{decomp3}
There is a bijective correspondence between walks of length $am$
starting at $0$ with 
a right-step and ending at $0$, and admissible compositions with
initial term $P_{0j_1}(m_1),\;0\leq j_1\leq a-2,\,m_1\geq 1,$ and  final term
$P_{00}(m_r),\,N_{00}(m_r),\; m_r\geq 1,$ or $T_{j_r0},\; 1\leq
j_r\leq a-2$, where $m_1+\dots+m_r=m$. 
\end{cly}

\section{Proof of Theorem~\ref{thm}}
\label{proofthm}
The main purpose of this section is to give a proof of
Theorem~\ref{thm} by exploiting the decomposition results of the
preceding section. 

For $m\geq 1$, we let $A_m$ denote the number of positive walks $P_{ii}(m)$,
which obviously is independent of $i\in\mathbb Z$ and also equals the
number of negative walks $N_{ii}(m)$. Moreover, $A_m$
also equals the number of right, respectively left, $s$-pyramids of
size $m$ as a consequence of  Lemma~\ref{stringpyr2}.

By Lemma~\ref{decomp2} it follows that the number $B_m$ of all
$0$-pyramids of size $m$ can be written as
\begin{equation}\label{AB}
B_m = \sum_{r\geq 1}\sum_{m_1+\dots+m_r=m} (a-1)^{r-1}A_{m_1}\dots A_{m_r},\,
\end{equation}
since the number of possible choices of the sequence $(s_1,\dots,s_r)$
in Lemma~\ref{decomp2} is $(a-1)^{r-1}$.

On the other hand, it follows from Corollary~\ref{decomp3} that the
number of walks of length $am$ starting at $0$ with a right-step and
ending at $0$ can be written in the form
$$
\sum_{r\geq 1}\sum_{m_1+\dots+m_r=m} a_r\,A_{m_1}\dots A_{m_r},\,
$$
where $r$ denotes the total number of $P$- and $N$-terms, with sizes
$m_1,\dots,m_r\geq 1$, in a 
composition and the factor $a_r$ counts the number of admissible compositions
subject to the boundary conditions specified in
Corollary~\ref{decomp3} for fixed $r$ and $m_1,\dots,m_r$. As
indicated, this number only depends on $r$. Of
course, the total number of walks of length $am$ starting at $0$ with a
right-step and ending at $0$ equals $am-1\choose m-1$. Thus the proof
of  Theorem~\ref{thm} will be completed by proving the following
lemma.
\begin{lem}\label{exp} For every $a\geq 3$ we have
$$ a_r = (a-1)^{r-1}\,,\;\quad r\geq 1\,.
$$
\end{lem} 

%

\begin{proof} We use standard matrix techniques.
Set $b:=a-1\geq 2$ in the following and define the $2b\times 2b$-matrices
${\cal E},\,{\cal T}$ and $\cal U$ by
\begin{equation*}
{\cal E} = {\bf E}\otimes \begin{pmatrix}0~~ 1\\ 1~~ 0 \end{pmatrix}\;,\qquad
{\cal T} = {\bf T}\otimes \begin{pmatrix}1~~ 1\\ 1~~ 1 \end{pmatrix}\;,\qquad
{\cal U} = {\bf U}\otimes \begin{pmatrix}0~~ 1\\ 0~~ 0\end{pmatrix}\;,
\end{equation*}
where ${\bf E}$ is the $b\times b$ unit matrix, ${\bf T}$
is the lower triangular  $n\times n$-matrix with $1$'s below the
diagonal and $0$'s 
elsewhere, and ${\bf U}$ is the transpose of ${\bf T}$. We
label the rows and columns of the first factor in the tensor products
by $i,j\in\{0,1,\dots,a-2\}$ and those of the second factor by
$R,S\in\{P,N\}$. Moreover, we shall use the ordering
$0P,\; 0N,\; 1P,\; 1N, \dots,\; (a-2)P,\; (a-2)N$ of double indices,
thus implying the
standard identification of the tensor product of a $b\times b$-matrix
and a $2\times 2$-matrix with a $2b\times 2b$-matrix.

By construction, ${\cal E}_{iR,jS}=1$ if and only if
$R_{ii}(m)S_{jj}(m')$ can occur in an admissible composition, i.e. if
$i=j$ and $R\neq S$ according to Definition~\ref{admis}. Similarly, 
${\cal T}_{iR,jS}=1$ if and only if $R_{ii}(m)T_{ij}S_{jj}(m')$ is
allowed, and ${\cal U}_{iR,jS}=1$ if and only if $R_{ii}(m)U_{ij}S_{jj}(m')$ is
allowed.  Viewing an admissible composition as a
chain of links of one of the forms
$R_{ii}(m)S_{jj}(m'),\;R_{ii}(m)T_{ij}S_{jj}(m')$ or
$R_{ii}(m)U_{ij}S_{jj}(m')$, a chain of
$r-1$ links contains a total of $r$ $P$- and $N$-terms.  Expanding
the power $({\cal E} + {\cal T}+{\cal U})^{r-1}$ we hence get
$$
a_r = \left\{({\cal E} + {\cal T}+{\cal U})^{r-1}\right\}_{0P,0P} +
\left\{({\cal E} + {\cal T}+{\cal U})^{r-1}\right\}_{0P,0N} +\sum_{j,S}
\left\{({\cal E} + {\cal T}+{\cal U})^{r-1}\right\}_{0P,jS}{\bf T}_{j0}\,,
$$
where the three terms correspond to the three possible types of final terms in
the compositions specified in Corollary~\ref{decomp3}. Since
${\bf T}_{j0}=1$ if $1\leq j\leq a-2$ and
${\bf T}_{00}=0$, the result can be rewritten as
$$
a_r = \begin{pmatrix}1~0~0\dots 0\end{pmatrix}({\cal E} + {\cal
T}+{\cal U})^{r-1}\begin{pmatrix}1\\1\\\vdots\\ 1\end{pmatrix}\,. 
$$

Before using this to evaluate $a_r$ in general it is instructive first to 
consider the case $a=3$ explicitly. In this case 
\begin{eqnarray*}
{\cal E} &=& \begin{pmatrix}1~~ 0\\ 0~~ 1 \end{pmatrix}\otimes
\begin{pmatrix}0~~ 1\\ 1~~ 0 \end{pmatrix} = \begin{pmatrix}0~~ 1~~
0~~ 0\\ 1~~ 0~~ 0~~ 0\\0~~ 0~~ 0~~ 1\\0~~ 0~~ 1~~ 0 \end{pmatrix} \\
{\cal T} &=& \begin{pmatrix}0~~ 0\\ 1~~ 0 \end{pmatrix}\otimes \begin{pmatrix}1~~ 1\\ 1~~ 1 \end{pmatrix} = \begin{pmatrix}0~~ 0~~
0~~ 0\\ 0~~ 0~~ 0~~ 0\\1~~ 1~~ 0~~ 0\\1~~ 1~~ 0~~ 0 \end{pmatrix}\\
{\cal U} &=& \begin{pmatrix}0~~ 1\\ 0~~ 0 \end{pmatrix}\otimes \begin{pmatrix}0~~ 1\\ 0~~ 0\end{pmatrix} = \begin{pmatrix}0~~ 0~~
0~~ 1\\ 0~~ 0~~ 0~~ 0\\0~~ 0~~ 0~~ 0\\0~~ 0~~ 0~~ 0 \end{pmatrix}
\end{eqnarray*} 
such that 
$$
a_r = \begin{pmatrix}1~ 0~ 0~ 0\end{pmatrix}\begin{pmatrix}0~~ 1~~
0~~ 1\\ 1~~ 0~~ 0~~ 0\\1~~ 1~~ 0~~ 1\\1~~ 1~~ 1~~ 0 \end{pmatrix}^{r-1}\begin{pmatrix}1\\1\\1\\ 1\end{pmatrix}\,. 
$$
The characteristic polynomial of the $4\times 4$-matrix entering this
expression is found to be
$$
p_2(\lambda) = \lambda(\lambda -2)(\lambda+1)^2
$$
and hence its eigenvalues are $2,\,0$ and $-1$. The eigenvalue
multiplicities are seen to be $1$ such that the matrix is not
diagonalisable.
 
In order to determine $a_r$ we define 
$$
v_r = \begin{pmatrix}a_r\\a'_r\\b_r\\b'_r\end{pmatrix} = \begin{pmatrix}0~~ 1~~
0~~ 1\\ 1~~ 0~~ 0~~ 0\\1~~ 1~~ 0~~ 1\\1~~ 1~~ 1~~ 0 \end{pmatrix}^{r-1}\begin{pmatrix}1\\1\\1\\ 1\end{pmatrix}\,, 
$$
which fulfills the recursion relations
\begin{eqnarray}
b_{r+1}-a_{r+1} &=& a_r \qquad~~~~~~  b'_{r+1}-a'_{r+1} = a'_r + b_r \label{rec1}\\
b'_{r+1}-b_{r+1} &=& b_r-b'_r \qquad~~~~~~~~~  a'_{r+1} = a_r\,.  \label{rec2}
\end{eqnarray}
for all $r\geq 1$. Since $b_1=b'_1=1$ we get from the first equation
in \eqref{rec2} that $b_r=b'_r$ for $r\geq 1$. The first equation
in \eqref{rec1} and the second in \eqref{rec2} imply
\begin{equation}\label{rec3}
b_r = a_r + a'_r\,,\quad r\geq 2\,.
\end{equation}
Combining this with the last equations in \eqref{rec1} and
\eqref{rec2} then gives
$b_{r+1} = 2b_r\,, r\geq 2\,,$
and hence
$$
b_r = b_2\cdot 2^{r-2} = 3\cdot 2^{r-2}\,,\quad r\geq 2\,.
$$
Inserting this into the first equation in \eqref{rec1} yields
$$
a_{r+1}+a_r = 3\cdot 2^{r-1}\,,\quad r\geq 1\,.
$$
The solution to this equation with initial condition $a_1=1$ is 
$$
a_r = 2^{r-1}\,,\quad r\geq 1\,,
$$
as desired.

We now return to the general case $a\geq 3$. Setting ${\cal A}={\cal
E} + {\cal T}+{\cal U}$ one can write ${\cal
A} = {\cal X} + {\cal Y}$ where ${\cal X}$ is the lower triangular
$2b\times 2b$-matrix whose matrix elements below the diagonal equal
$1$ and are otherwise $0$, and ${\cal B}$ is the upper triangular
matrix with matrix elements equal to $1$ in slots with even row and
column indices above the diagonal and $0$ elsewhere, that is 
$$ 
{\cal A} = \begin{pmatrix} 0~~ 1~~ 0~~ 1~~ 0~~ 1 \dots 1~~ 0~~ 1\\
 1~~ 0~~ 0~~ 0~~ 0~~ 0 \dots 0~~ 0~~ 0\\
1~~ 1~~ 0~~ 1~~ 0~~ 1 \dots 1~~ 0~~ 1\\
1~~ 1~~ 1~~ 0~~ 0~~ 0 \dots 0~~ 0~~ 0\\
\vdots~~~~~~~~~~~~~~~~~~~~~~~~~~~~~~\vdots\\
1~~ 1~~ 1~~ 1~~ 1~~ 1 \dots 1~~ 0~~ 1\\
1~~ 1~~ 1~~ 1~~ 1~~ 1 \dots 1~~ 1~~ 0\end{pmatrix}
$$
It is now easy to show by induction w.r.t. $b$ that the characteristic
polynomial $p_b$ of ${\cal A}$ is given by
$$
p_b(\lambda) = \lambda^{b-1}(\lambda -b)(\lambda +1)^b\,.
$$
Hence the eigenvalues of ${\cal A}$ are $0, b$ and $-1$, and it is
readily seen that they all have eigenvalue multiplicity equal to
$1$. We denote by $e$ the eigenvector with eigenvalue $b$ normalised
such that its first coordinate is $b$. One finds
$$
e = (\,b\,,\; 1\,, \;b\zeta\,, \;2\zeta\,, \;b\zeta^2, \;3\zeta^2,\;\dots\;,
\;b\zeta^{b-2},\;(b-1)\zeta^{b-2}, \;b\zeta^{b-1},\;b\zeta^{b-1}\,)\,,
$$
where
$$
\zeta = 1+b^{-1}\,.
$$

Next, define 
\begin{equation}\label{minusone}
f_i := (0,\;\dots\;,\;0,\;-i,\;1,\;1,\;\dots\,,1)\,,\quad 1\leq i\leq
b-1\,,
\end{equation}
where $-i$ is the $2(b-i)$'th coordinate (such that the number of $1$'s is
$2i$). Then $f_1$ belongs to the kernel of ${\cal A}$ and one finds by
direct computation that
\begin{equation}\label{minusto}
{\cal A} f_i = f_1 + f_2 + \dots + f_{i-1}\,,\quad 2\leq i\leq b-1\,.
\end{equation}
It follows that the kernel of ${\cal A}^{b-1}$ is spanned by the
vectors $f_1,\;\dots\;,f_{b-1}$. By a straight-forward calculation one
finds
$$
(1,\;1,\;1,\;\dots\; 1) = b^{-1}\left(e - \sum_{i=1}^{b-1}
\zeta^{b-1-i} f_i\right)\;.
$$
From this we conclude that
\begin{eqnarray*}
a_r &=& b^{-1}\begin{pmatrix}1~0~0\dots 0\end{pmatrix}{\cal
A}^{r-1}\left(e - \sum_{i=1}^{b-1}\zeta^{b-1-i} f_i\right)\\
&=&   b^{-1}\begin{pmatrix}1~0~0\dots 0\end{pmatrix}{\cal
A}^{r-1}\,e\\
&=& b^{-1}b^{r-1}\,\begin{pmatrix}1~0~0\dots 0\end{pmatrix}e\\ &=& b^{r-1}\qquad \mbox{for $r\geq 1$},
\end{eqnarray*}
where, in the second step, we have used \eqref{minusone} and
\eqref{minusto} and the fact that $f_i$ has vanishing first coordinate
for all $i=1,\dots,b-1$. 

Recalling the definition of $b$ the lemma is proven.
\end{proof}

\begin{rem}
As previously mentioned, the method used in this section to determine
the number $B_m$ of pyramids 
of size $m$ as given by Theorem~\ref{thm} did not require knowing the
number $A_m$ of right-pyramids of size $m$, as given by
Corollary~\ref{cor}.  For $a=2$ the use of generating function
techniques as in \cite{bousrech} proceeds by first determining the
generating function $A(t)$ for the $A_m$ and 
then using a simple algebraic relation between this function and the
generating function  
$$
 B(t) =\sum_{m=1}^\infty B_m\, t^m
$$
for the $B_m$. For general $a\geq 2$, this relation is a special case
of eq. \eqref{A-B} below and takes the form
\begin{equation}\label{A-B0}
B(t) = \frac{A(t)}{1-(a-1)A(t)}\,,
\end{equation}
For $a=2$ the quadratic relation \eqref{recursion} satisfied by $A(t)$
has a simple 
solution which easily yields the $B_m$ in closed form when inserted
into \eqref{A-B0}.
 For $a>2$ such a procedure does not seem feasible.
\end{rem}

\bigskip

We conclude this section by determining the asymptotic behaviour of the
average width of pyramids of large size. For this purpose we first
note that the asymptotic behaviour of 
$$
B_m = {am-1\choose m-1}
$$ 
is readily obtained from Stirling's formula and is given by
\begin{equation}\label{Basymp}
B_m\; \sim\; \frac{1}{\sqrt{2\pi a(a-1)
m}}\left(\frac{a^a}{(a-1)^{a-1}}\right)^m\,,\qquad m\to\infty\,.
\end{equation}
Next, let $B_{m,n}$ denote the number of pyramids of size $m$ and
 left width $n$, where the \emph{left width} of a pyramid $p$ with bottom piece
covering the interval $]0,a[$ equals $n$ if the leftmost interval
covered by a piece in $p$ is $]-n,a-n[$, and let
$$
B(t,v) = \sum_{n\geq 0,m\geq 1} B_{n,m} t^mv^n
$$ 
be the corresponding generating function. Now recall the decomposition,
 in the proof of \linebreak Lemma~\ref{decomp2}, of a pyramid $p$
into a right $0$-pyramid $p_1$ and an arbitrary pyramid $p'$ with bottom
piece covering the interval $]s_2-a,s_2[$, where $1\leq s_2\leq
a-1$, or $p'$ may be empty. This is seen to imply the relation
$$
 B(t,v) = A(t)(1+ (v+v^2+\dots+v^{a-1})B(t,v))
$$
between $A(t)$ and $B(t,v)$, that is 
\begin{equation}\label{A-B}
B(t,v) = \frac{A(t)}{1-(v+v^2+\dots+v^{a-1})A(t)}\,.
\end{equation} 
In particular, the relation \eqref{A-B0} is obtained for $v=1$.
 
Moreover, differentiating eq. \eqref{A-B} with respect to $v$ and setting
$v=1$ we get 
 \begin{equation}\label{B-C}
C(t) = \frac 12 a(a-1)B(t)^2 \,.
\end{equation}
where $C(t)$ is the generating function with coefficients
$$
C_m = \sum_{n\geq 0} n\,B_{m,n}\,.
$$
The asymptotic behaviour of $C_m$ can now be obtained by standard
singularity analysis. Indeed, from the polynomial equation
\eqref{recursion} satisfied by $A(t)$ we conclude that the
singularity of $A(t)$ closest to the origin is at 
$$
t_0 = \frac{(a-1)^{a-1}}{a^a}
$$
and that 
$$
A(t_0)-A(t)\;\sim\; c_0 \left(1-\frac{t}{t_0}\right)^{\frac
12}\,,\qquad t\to t_0\,,
$$
where
$$
c_0 = (a-1)^{-2}\sqrt{2a(a-1)}\quad\mbox{and}\quad A(t_0) = (a-1)^{-1}\,.
$$    
Using \eqref{A-B0} and \eqref{B-C} one obtains
$$
C(t)\;\sim\; \frac 14\, \left(1-\frac{t}{t_0}\right)^{-1}\,,\qquad
t\to t_0\,,
$$ 
and hence \cite{flaod}
\begin{equation}\label{Casymp}
C_m\;\sim\; \frac 14\, t_0^{-m} = \frac 14\,\left(
\frac{a^a}{(a-1)^{a-1}}\right)^m \,,\qquad m\to\infty\,. 
\end{equation}
For the average left width with respect to the uniform distribution of pyramids
of size $m$ we conclude from \eqref{Basymp} and \eqref{Casymp} that
$$
\frac{C_m}{B_m}\;\sim\; \frac 14 \sqrt{2\pi a(a-1)m}\,,\qquad m\to\infty\,.
$$
Hence we have proven the following result for the average width which
equals twice the average left width plus $a$.
\begin{prop}\label{averasymp}
The average width of pyramids of size $m$ is asymptotic to 
$$
 \sqrt{\frac{\pi}2 a(a-1)m}
$$
for $m\to\infty$.
\end{prop}

\bigskip

\section{Concluding remarks}
As noted in the introduction the pyramids under consideration in this
paper may be considered as special planar LEGO structures built from
$1\times a$ LEGO pieces. More specifically, consider the number
$L_m^a$ of contiguous LEGO structures made out of $m$ $1\times a$
pieces which are "flat" in the sense that all pieces are contained in
the same vertical plane, and such that there is a unique piece in the
lowest level of the structure. Pictorially speaking, the difference
between that case and the one studied here is that pieces are allowed
to hang underneath other pieces from the second level of the structure
and upwards. These numbers turn out to be hard to compute; in
\cite{abeil1} some of them have been calculated for $1\leq
a\leq 8$ and are reproduced in The Online Encyclopaedia of Integer
Sequences \cite{encseq}. 

In \cite{dureil} we have shown, using standard
concatenation arguments, that the exponential growth rate 
$$
g_a = \lim_{m\to\infty}\frac{\ln L_m^a}{m}
$$
is well defined and finite (and that this, in fact, holds also for more
general non-planar classes of LEGO structures). Obviously, the
asymptotic relation \eqref{Basymp} gives  
\begin{equation}\label{lowasymp}
h_a\,\geq\, \frac{a^a}{(a-1)^{a-1}}\;\sim\;  e^{-1}(a-1)
\end{equation}
for large $a$, where we have set $h_a=e^{g_a}$.
This lower bound appears to be rather tight; it
is, indeed, at present our best lower bound for general $a$. It
may be improvable by known techniques for fixed and relatively small
values of $a$. For instance, for $a=2$ one can appeal to the
enumeration of multi-pyramids in \cite{bousrech} to get $h_2\geq
9/2$. By adapting the method of \cite{dureil} to this setting we can
improve this lower estimate further to $h_2\geq 4.607$ by computing
the number $c_m$ of "fat" structures up to level $16$ and proving
$c_{m+2}\geq 5c_m$ for all $m$.

Upper bounds on $h_a$ can be produced by adapting the method of
\cite{klariv} to this setting. Performing an analysis of depth 1 we
can prove that, for all $a$, the largest root of 
\begin{align*}
(\tfrac 14 a^9 -\tfrac 45 a^8 +\tfrac{21}{8}a^7 -3 a^6 +2 a^5 -\tfrac 34
a^4+\tfrac 18 a^3)x^5 +\\
(-3a^8 +\tfrac{77}{4}a^7 -\tfrac{105}{2}a^6
+\tfrac{159}{2}a^5-73a^4+\tfrac{165}{4}a^3 -\tfrac{27}{2}a^2 +2a)x^4+\\
(-\tfrac{47}{8}a^7+27a^6
-\tfrac{195}{4}a^5+\tfrac{85}{2}a^4-\tfrac{135}{8}a^3+\tfrac 32 a^2+\tfrac
12 a)x^3+\\
(a^6 -4a^5+6a^4-4a^3+a^2)x^2
\end{align*}
is a majorant of $h_a$, which in turn shows that
$$
h_a\;\leq\; 6.356\,a - 4.375
$$
for large $a$. No closed form upper bound is available for analyses of
depth 2 and 3, but majorants are readily computable for $a$ up to $8$
as indicated  Figure \ref{lego}. Again, these appear to be
approximately affine for large $a$. 

Using the Monte Carlo methods described in \cite{abeil} and \cite{abeil1} we have
produced estimates of $h_a$ for $a$ up to $8$. Strikingly, our
estimates in each case have the form $k_a\frac{a^a}{(a-1)^{a-1}}$, 
with $k_a$ between $1.238$ and $1.264$. This makes it tempting to
speculate that
$$
h_a = \frac{5\,a^a}{4(a-1)^{a-1}}\,.
$$
In
particular, it does not seem unlikely that $h_2=5$. Our best current estimate,
achieved by a least square fitting of a function of the form $AH^nn^C$
with Monte Carlo estimates for $L^2_{16},\dots,L^2_{20}$ yields
$H=5.0012$.

\begin{figure}
\begin{center}
\includegraphics[width=10cm]{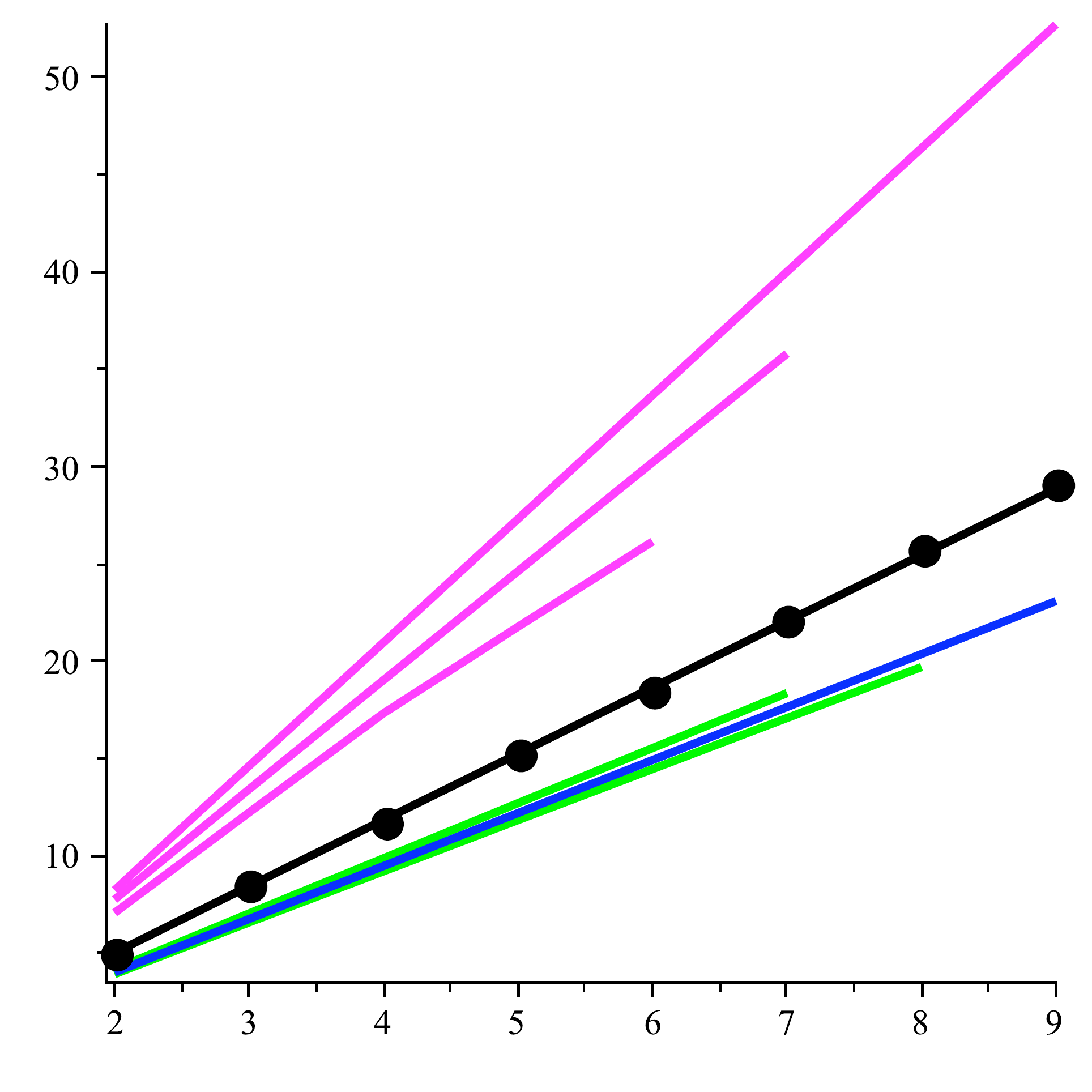}\label{lego}
\end{center}
\caption{Black: our conjectural formula for $h_a$ and the estimated
values. Magenta: Upper bounds based on the Klarner-Rivest
method (levels 1, 2 and 3). Blue: Lower bound from pyramids. Green: Lower bounds from
counting fat buildings (levels 6 and 8).}
\end{figure}

\end{document}